\documentclass{amsart}
\usepackage{amssymb}
\usepackage{amsmath}
\usepackage{amsthm}
\usepackage{amsfonts}
\usepackage{url}
\usepackage{enumerate}
\theoremstyle{plain}
\newtheorem{theorem}{Theorem}[section]
\newtheorem*{theorem*}{Theorem}
\newtheorem{lemma}[theorem]{Lemma}
\newtheorem{proposition}[theorem]{Proposition}
\newtheorem{conjecture}[theorem]{Conjecture}
\newtheorem{fact}[theorem]{Fact}
\newtheorem*{fact*}{Fact}
\newtheorem{corollary}[theorem]{Corollary}

\newtheorem*{claim*}{Claim}

\theoremstyle{definition}
\newtheorem{definition}[theorem]{Definition}

\newtheorem*{definition*}{Definition}

\newtheorem*{notation*}{Notation}

\theoremstyle{remark}
\newtheorem{remark}[theorem]{Remark}
\newtheorem*{remark*}{Remark}
\newtheorem{examples}[theorem]{Examples}
\newtheorem{example}[theorem]{Example}
\newtheorem*{examples*}{Examples}
\newtheorem{question}[theorem]{Question}
\newtheorem*{question*}{Question}

\newtheorem*{note*}{Note}

\newcommand{\Z}{\mathbb{Z}}
\newcommand{\N}{\mathbb{N}}
\newcommand{\Q}{\mathbb{Q}}
\newcommand{\C}{\mathbb{C}}

\newcommand{\F}{\mathbb{F}}

\newcommand{\rf}{{\mathbf k}}
\newcommand{\RV}{{\mathrm{\bf RV}}}
\newcommand{\vf}{{\mathbf{K}}}
\newcommand{\vg}{{\mathbf \Gamma}}

\newcommand{\IncVFA}{\mathrm{IncVFA}_0}

\newcommand{\IncODG}{\mathrm{IncODG}}
\newcommand{\IncDODG}{\mathrm{IncDODG}}
\newcommand{\ACFA}{\mathrm{ACFA}}
\newcommand{\ACVF}{\mathrm{ACVF}}
\newcommand{\VFA}{\mathrm{VFA}_0}

\renewcommand{\O}{\mathcal{O}}

\newcommand{\res}{\mathrm{res}}

\newcommand{\ac}{\mathrm{ac}}

\newcommand{\rk}{\mathrm{rk}} 
\newcommand{\val}{\mathrm{val}} 
\newcommand{\m}{\mathfrak{m}}
\renewcommand{\a}{\overline{a}}
\renewcommand{\b}{\overline{b}}

\newcommand{\NTP}{\mathrm{NTP}}
\newcommand{\Th}{\mathrm{Th}}
\newcommand{\TP}{\mathrm{TP}}
\newcommand{\NIP}{\mathrm{NIP}}

\newcommand{\weight}{\mathop{w}}

\newcommand{\td}{\mathrm{td}}
\newcommand{\order}{\mathrm{order}}
\newcommand{\complexity}{\mathrm{complexity}}
\newcommand{\Fix}{\mathrm{Fix}}
\newcommand{\tp}{\mathrm{tp}}
\newcommand{\qftp}{\mathrm{qftp}}

\newcommand{\car}{\mathrm{char}}
\newcommand{\indu}{\mathrm{ind}}

\newcommand{\elex}{\preccurlyeq}

\renewcommand{\L}{{\cal L}}
\renewcommand{\phi}{\varphi}
\newcommand{\LODG}{\L_{ODG}}

\newcommand{\cal}{\mathcal }

\def\Ind#1#2{#1\setbox0=\hbox{$#1x$}\kern\wd0\hbox to 0pt{\hss$#1\mid$\hss}
\lower.9\ht0\hbox to 0pt{\hss$#1\smile$\hss}\kern\wd0}
\def\Notind#1#2{#1\setbox0=\hbox{$#1x$}\kern\wd0\hbox to 0pt{\mathchardef
\nn="3236\hss$#1\nn$\kern1.4\wd0\hss}\hbox to 0pt{\hss$#1\mid$\hss}\lower.9\ht0
\hbox to 0pt{\hss$#1\smile$\hss}\kern\wd0}

\global\long\def\M{\operatorname{\mathbb{M}}}

\title{Valued difference fields and $\NTP_2$}
\author{Artem Chernikov}
\address{The Hebrew University of Jerusalem \\ Einstein Institute of Mathematics \\ Edmond J. Safra Campus, Givat Ram \\ 91904 Jerusalem, Israel}
\email{art.chernikov@gmail.com}
\thanks{The first author was supported by the Marie Curie Initial Training Network in Mathematical Logic - MALOA - From 
MAthematical LOgic to Applications, PITN-GA-2009-238381. }

\author{Martin Hils} 
\address{Institut de Math\'ematiques de Jussieu (UMR 7586 du CNRS), 
Universit\'e Paris Diderot Paris 7, UFR de Math\'ematiques - case 7012, 75205 Paris Cedex 13, France, and\newline\indent 
\'Ecole Normale Sup\'erieure, D\'epartement de math\'ematiques et applications, 45 rue d'Ulm, 75230 Paris Cedex 05, France (UMR 8553 du CNRS).}
\email{hils@math.univ-paris-diderot.fr}
\thanks{The second author was partially funded by the Agence Nationale de Recherche [MODIG, Projet ANR-09-BLAN-0047]. \\ 
\indent A substantial part of the research was carried out during the program 'Model Theory and Applications' at  the Max Planck Institut f\"{u}r Mathematik in Bonn. The 
second author would like to thank for the hospitality and the support of the MPIM}


\date{\today}

\begin{document}

\bibliographystyle{halpha}

\begin{abstract}
We show that the theory of the non-standard Frobenius automorphism, acting on an algebraically closed valued field of equal characteristic 0, is $\NTP_2$. More generally, in the contractive as well as 
in the isometric case, we prove that a $\sigma$-henselian 
valued difference field of equicharacteristic 0 is $\NTP_2$, provided both the residue difference 
field and the value group (as an ordered difference group) are $\NTP_2$. 
\end{abstract}

\maketitle


\section{Introduction}

Model theory has proven to be a fruitful framework to study fields with extra structure. Central examples include valued fields (e.g.\ the work of Haskell, Hrushovski and Macpherson on 
the theory of algebraically closed valued fields, $\ACVF$, see \cite{HHMStableDomination}) and difference fields (starting with the work of Chatzidakis and Hrushovski on the theory 
of algebraically closed fields with a \emph{generic} automorphism, $\ACFA$, see \cite{ChHr99}). In this paper, we are considering a combination of these, namely valued difference fields, 
i.e.\ valued fields with a distinguished automorphism (preserving the valuation ring).

Every non-principal ultraproduct of structures of the form $({\mathbb{F}}_p^a,Frob_p)$ is a model 
of $\ACFA_0$, i.e. the non-standard Frobenius is a generic automorphism. This is a deep 
result of Hrushovski \cite{Hru04} which required a twisted version of the Lang-Weil estimates.

One may consider the non-standard Frobenius acting on an algebraically closed \emph{valued} field, i.e.\ 
the limit theory of the Frobenius automorphism acting on an algebraically closed valued field of 
characteristic $p$ (where $p$ tends to infinity). Hrushovski 
\cite{Hru02} gives a natural axiomatisation of this limit theory in the language of valued difference fields (denoted by $\VFA$ in the sequel).
Durhan (formerly Azg{\i}n) \cite{Azg10} obtains an alternative axiomatisation, as well as an Ax-Kochen-Ershov 
principle for a certain class of valued difference fields.

The theory $\VFA$ is interesting from an algebraic point of view. The residue field together with the 
induced automorphism $\overline{\sigma}$ is a model of $\ACFA_0$, by the aforementioned result of Hrushovski. 
The induced automorphism $\sigma_{\Gamma}$ on the value group $\Gamma$ is \emph{$\omega$-increasing} (i.e. $\sigma_{\Gamma}(\gamma)>n\gamma$ for all $\gamma>0$ and $n\geq1$; 
valued difference fields satisfying this property will be called \emph{contractive}). 
Thus, $\Gamma$ gets the structure of a divisible torsion free ordered $\mathbb{Z}[\sigma]$-module (i.e. an ordered 
vector space over $\mathbb{Q}(\sigma)$, where $\sigma\gg1$ is an indeterminate). 
It is sufficient to add a \emph{$\sigma$-Hensel property} (see Definition \ref{D:sigma-H}) to obtain an axiomatisation of $\VFA$. 

Moreover, valued difference fields enter the study of (non-valued) difference fields by way 
of \emph{transformal specialisations} (see \cite{Hru04}). A better understanding  of valued difference fields 
will most probably shed new light on Hrushovski's proof of the non-standard Frobenius result. 

\

Work of Shelah on the model theoretic classification program \cite{MR1083551} had demonstrated the importance of understanding which combinatorial configurations a theory can encode. In the case of stable theories (i.e. theories that cannot encode linear order) he had developed a beautiful and fruitful theory of analysing types and models. Further work of Poizat, Hrushovski and other researchers, generalising the ideas of Zilber in the finite rank case, culminated in the creation of \emph{geometric stability theory} establishing deep connections between the geometry of forking independence and properties of algebraic structures (groups and fields) definable in the theory.
Later on it became clear that stability-theoretic methods can be generalised to larger contexts, and in the last twenty years there had been two main directions: simple theories  \cite{WagnerBook} and, more recently, NIP theories \cite{AdlNIP,Sim12}. A characteristic property of these developments is that the motivation is coming both from purely model theoretic considerations and  from the study of particular important algebraic structures: $\ACFA$ as a prototypical example of a simple unstable theory, and $\ACVF$ as a typical example of an (unstable) $\NIP$ theory. These lines of research had found numerous applications \cite{HruMM, HruKazh, HruLoes}.

Observe that the theory $\VFA$ is neither \emph{simple} (due to the total order in the value group) nor \emph{$\NIP$} (due to the independence property which holds in the residue 
field). 

It turns out that in the 80's Shelah had defined another class ---  $\NTP_2$ theories, or theories without the tree property of the second kind \cite{MR595012}. This class generalises both simple and $\NIP$ theories, and contains new examples (e.g. any ultraproduct of $p$-adics is $\NTP_2$ \cite{Che12}). Recently it had attracted attention, largely motivated by Pillay's question on equality of forking and dividing over models in $\NIP$, and a theory of forking for $\NTP_2$ theories  had been developed  \cite{CheKap, Che12, CheBY}.

In this paper we show the following general theorem.

\begin{theorem*}[\ref{T:Main}]
Let ${\cal K}=(K,k,\Gamma, v, \ac)$ be a valued difference fields of residue characteristic 0 (where $\ac$ is an angular component map commuting with $\sigma$). Assume that $T=\mathrm{Th}({\cal K})$ eliminates $\vf$-quantifiers, and that both the residue field $k$ (as a difference field) and the value group $\Gamma$ (as an ordered difference group) are $\NTP_2$. 
Then, ${\cal K}$ is $\NTP_2$.
\end{theorem*}

Our method of proof combines a new result on extending indiscernible arrays by parameters 
coming from $\NTP_2$ sorts with a back-and-forth system coming from the elimination of field quantifiers. In this way, we reduce the statement to a situation where one 
deals with immediate extensions, and these extensions are controlled by $\NIP$ formulas.

Applying the theorem we obtain new and interesting algebraic examples of $\NTP_2$ theories:
\begin{itemize}
\item $\VFA$ is $\NTP_2$. More generally, a contractive $\sigma$-henselian valued difference field of equicharacteristic 0 
(where an Ax-Kochen-Ershov principle holds by \cite{Azg10}) is $\NTP_2$, provided both the theory of the 
value group (with the induced automorphism) and the theory of the residue field (with the induced 
automorphism) are $\NTP_2$.

\item We prove a similar result in the isometric case, where an Ax-Kochen-Ershov principle holds 
as well \cite{Sca03,BeMaSc07,AzVa11}.

\end{itemize}

Similar transfer results hold in the context of valued fields (i.e. when $\sigma$ is the identity), where they are derived from the usual Ax-Kochen-Ershov 
principle: Delon \cite{Del81} showed this for $\NIP$, Shelah \cite{She09} for strongly dependent (a strengthening of $\NIP$) and 
the first author \cite{Che12} for $\NTP_2$ (and finiteness of burden).

A quick overview of the paper. In Section \ref{S:Prelim}, we recall the basic model-theoretic results on valued difference fields and elimination of field quantifiers. Section \ref{S:NTP2} 
contains general results on manipulating indiscernible arrays in $\NTP_2$ theories. The main theorem of the 
paper is then proved in Section \ref{sec: general theorem}, and applications are given in Section \ref{sec: applications of the main theorem}. We end with a list of open problems and some related observations, in particular discussing when ordered modules are $\NIP$.

\subsection*{Acknowledgements}
We are grateful to the referee for numerous corrections and suggestions on improving the paper.

\section{Preliminaries on valued difference fields} \label{S:Prelim}
In this section, we present the necessary material on valued difference fields which we will need for our purposes. 

Moreover, we will give a 
survey of the results in the contractive case, in particular with respect to $\VFA$, the theory which motivated our study. Strictly speaking, these 
results are not needed in the main theorem. Nevertheless, we believe this is useful for the reader, as the results in question are not so widely known and easily accessible. Most of the material in the contractive case may be found in \cite{Azg10}, but we will also need the context from \cite{Azg07}, where an angular component is used instead of a cross-section. See also \cite{Pal10,Gia11}, and the unpublished notes \cite{Hru02} of Hrushovski, where 
most of the ideas in the contractive case already appear. 

At the end of the section, we will briefly mention the isometric case which historically preceded the contractive case \cite{Sca00,Sca03,BeMaSc07,AzVa11}.

\subsection{Ordered difference groups}\label{SubS:ODG}\label{Sub:ODG}
An \emph{ordered difference group }is a structure of the form $\langle \Gamma,0,+,-,<,\sigma\rangle$, where $\langle\Gamma,0 ,+,-,<\rangle$ is an ordered 
abelian group and $\sigma$ is an automorphism of $\langle\Gamma,0 ,+,-,<\rangle$. The automorphism $\sigma$ is called \emph{$\omega$-increasing} 
if $\sigma(\gamma)> n\gamma$ for all $\gamma\in\Gamma_{>0}$ and all natural numbers $n$; the corresponding difference group will also be called $\omega$-increasing. We treat ordered difference groups as first order structures in the 
language $\LODG=\{0,+,-,<,\sigma\}$; the class of $\omega$-increasing ordered difference groups may be axiomatised, and we denote it by 
$\IncODG$. 

Any $\langle \Gamma,+,<,\sigma\rangle\models\IncODG$ is an ordered (and so in particular torsion free) $\Z[\sigma, \sigma^{-1}]$-module, where $\Z[\sigma]$ is the ordered ring of polynomials 
in the indeterminate $\sigma$ with $\sigma\gg1$. For $p=p(\sigma)=\sum z_i\sigma^{i}\in\Z[\sigma, \sigma^{-1}]$ and $\gamma \in \Gamma$, one puts $p\cdot \gamma:=\sum z_i\sigma^{i}(\gamma)$. 
Conversely, any ordered $\Z[\sigma,\sigma^{-1}]$-module gives rise to a model of $\IncODG$. When divisible, such modules correspond to ordered vector spaces 
over the (ordered) fraction field $\Q(\sigma)$ of $\Z[\sigma]$. The theory of non-trivial divisible ordered $\Z[\sigma, \sigma^{-1}]$-modules will be denoted by $\IncDODG$.

The following fact is easy (see e.g.\ \cite{Pal10}).

\begin{fact}\label{F:ValInc}
The theory $\IncDODG$ is the model-completion of $\IncODG$. In particular, $\IncDODG$ eliminates quantifiers and is $o$-minimal.
\end{fact}

For $\gamma\in\Gamma\models\IncODG$ and $\zeta=(z_0,\ldots,z_n)\in\Z^{n+1}$, we will sometimes denote $\sum_{i=0}^{n}z_i\sigma^{i}(\gamma)$  by 
$\sigma^{\zeta}(\gamma)$.

\subsection{Valued difference fields}\label{Sub:ValDiffFields}\mbox{}

\noindent
{\it Notation and conventions.}\\
By a \emph{difference field }we will always mean a field $K$ together with a distinguished automorphism $\sigma$, i.e.\ what is sometimes called an 
\emph{inversive }difference field.

\smallskip

If $K$ is a difference field, one may form the ring of difference polynomials $K[X]_{\sigma}:=K[X,\sigma(X),\sigma^2(X),\ldots]$. Then $\sigma$ extends naturally to an endomorphism of 
$K[X]_{\sigma}$, and in this way $K[X]_\sigma$ is a difference ring extension of $K$.

\smallskip
If $K\subseteq L$ is an extension of difference fields and $a$ is a tuple from $L$, then $K\langle a\rangle$ denotes the difference 
field generated by $a$ over $K$; as a field, it is given by $K(\sigma^z(a),\, z\in\Z)$. An element $a\in L$ is called \emph{$\sigma$-algebraic} over $K$ 
if $g(a)=0$ for some non-constant $g(X)\in K[X]_{\sigma}$; else, it is called \emph{$\sigma$-transcendent} over $K$. 

\medskip

Recall that a \emph{valued field} is given by a surjective map $\val:K\rightarrow\Gamma_{\infty}$, where $K$ is a field and $\Gamma_{\infty}=\Gamma\cup\{\infty\}$, with $\Gamma$ an ordered abelian group and 
$\infty$ a distinct element satisfying
\begin{itemize}
\item $\val(x)=\infty\iff x=0$;
\item $\val(xy)=\val(x)+\val(y)$ for all $x,y\in K$;
\item $\val(x+y)\geq\min\{\val(x),\val(y)\}$ for all $x,y\in K$.
\end{itemize}
Here, the order is extended to a total order on $\Gamma_{\infty}$ making $\infty$ the maximal element, and the addition is extended so that $\infty$ becomes an absorbing element.

We will usually not distinguish between $\Gamma$ and $\Gamma_{\infty}$ and suppress $\infty$ in our paper.

The \emph{valuation ring} is given by $\O=\{x\in K\mid \val(x)\geq0\}$. It is a local ring, with maximal ideal $\m=\{x\in K\mid \val(x)>0\}$. The \emph{residue map }is 
given by $\res:\O\rightarrow \O/\m=: k$, and $k$ is called the \emph{residue field} of $K$. Sometimes, we will use $\a$ instead of $\res(a)$. We often write $\Gamma_K$ or $k_K$ to stress that we deal with the 
value group or residue field of the valued field $K$. An extension $K\subseteq L$ gives rise to extensions $k_K\subseteq k_L$ and $\Gamma_K\subseteq\Gamma_L$. 

\smallskip

A \emph{valued difference field }is a valued field $K$ together with a distinguished automorphism $\sigma$ satisfying $\sigma(\O)=\O$. Note that $\sigma$ induces 
an automorphism $\overline{\sigma}$ of the residue field, making it a difference field. Similarly, $\sigma$ induces an 
automorphism $\sigma_\Gamma$ of the value group, making it an ordered difference group. 
Most of the time, we will drop the subscript and use $\sigma$ for the automorphism on the value groups as well.

\smallskip

We treat valued difference fields in the three-sorted language $\L_{\rf,\vg,\sigma}$, consisting of 
\begin{itemize}
\item the language of difference rings $\L_{\vf}=\{0,1,+,-,\times,\sigma\}$ on the valued field sort denoted by $\vf$;
\item (a copy of) the language of difference rings  $\L_{\rf}=\{0,1,+,-,\times,\overline{\sigma}\}$ on the residue field sort denoted by $\rf$;
\item the language of ordered difference groups (with an additional infinite element) $\{0,+,-,<,\infty,\sigma_{\Gamma}\}$ on the value group sort denoted by $\vg$, and
\item the functions $\val:\vf\rightarrow\vg$ and $\res:\vf\rightarrow\rf$ between the sorts. (When considering a valued field as an $\L_{\rf,\vg,\sigma}$-structure, 
we may make the function $\res$ total by sending elements of negative valuation to $0\in\rf$.)
\end{itemize}

\medskip

An \emph{$\ac$-valued difference field} is a valued difference field $\mathcal{K}=(K,\Gamma,k,\sigma)$ together with an \emph{angular component map} $\ac:K\rightarrow k$ satisfying the following three properties:

\begin{itemize}
\item $\ac(x)=0$ iff $x=0$;
\item $\ac\upharpoonright_{K^{\times}}:K^{\times}\rightarrow k^{\times}$ is a group homomorphism commuting with $\sigma$;
\item for all $x\in K$ with $\val(x)=0$, one has $\ac(x)=\res(x)$.
\end{itemize}

We treat $\ac$-valued difference fields in the three-sorted language 
$\L_{\rf,\vg,\sigma}\cup\{\ac\}$. Note that the corresponding language without $\sigma$, $\overline{\sigma}$ and $\sigma_{\Gamma}$, denoted by 
$\L_{\rf,\vg}\cup\{\ac\}$, is precisely the \emph{language of Pas}.

\smallskip

If $A$ is a substructure of ${\cal K}=(K,\Gamma_K,k_K)$, we write $\vf(A)$ for the elements of $A$ which are in sort $\vf$. Similarly, we 
have $\vg(A)\subseteq \Gamma_K$ and $\rf(A)\subseteq k_K$. Note that in general $\val(\vf(A))$ is a proper subset of $\vg(A)$, and similarly for $\res$ and $\ac$.

\subsection{Elimination of field quantifiers in $\ac$-valued difference fields}\label{Sub:Elim-K-ac}
We gather here some useful consequences of the elimination of field quantifiers in 
$\ac$-valued difference fields. We may thus treat various cases in one common framework, 
namely $\sigma$-henselian valued difference fields, both in the contractive and in the 
isometric case, and henselian valued fields (without distinguished isomorphism).

\smallskip

The following lemma is a consequence of compactness, taking into account that 
there are no function symbols in our language with arguments in $\vg$ or $\rf$ and target sort $\vf$. 

\begin{lemma}\label{L:Back-and-forth-QE}
Let $T$ be an $\L_{\rf,\vg,\sigma}\cup\{\ac\}$-theory. The following are equivalent:
\begin{enumerate}
\item $T$ eliminates $\vf$-quantifiers.
\item Let $M$ and $M'$ be models of $T$, with substructures $A=({\vf}(A),\vg(A),\rf(A))\subseteq M$ and 
$A'=({\vf}(A'),\vg(A'),\rf(A'))\subseteq M'$. Let $f=(f_{\vf},f_{\vg},f_{\rf}):A\cong 
A'$ be an isomorphism such that 
\begin{itemize}
\item $f_{\vg}:\vg(A)\rightarrow \vg(A')$ 
is an $\{0,+,-,<,\infty,\sigma_{\Gamma}\}$-elementary map, and 
\item $f_{\rf}:\rf(A)\rightarrow \rf(A')$ 
is an $\L_{\rf}$-elementary map.
\end{itemize}
Then $f$ is an elementary map.
\end{enumerate}
\end{lemma}

\begin{lemma}\label{L:QE-consequences}
Let $T$ be an $\L_{\rf,\vg,\sigma}\cup\{\ac\}$-theory of $\ac$-valued difference fields. Assume that $T$ 
eliminates $\vf$-quantifiers. Then the 
following holds:

\begin{enumerate}
\item In any model ${\cal K}=(K,\Gamma_K,k_K)\models T$, $k_K=\rf({\cal K})$ is stably embedded, and the induced structure is that of a difference field. Similarly, $\Gamma_K=\vg({\cal K})$ is 
stably embedded and is a pure ordered difference group. Moreover, $\rf$ and $\vg$ are orthogonal, i.e.\ every definable subset of $\rf^m\times\vg^n$ is a finite union of rectangles.
\item Let ${\cal K}=(K,\Gamma_K,k_K)$ and ${\cal L}=(L,\Gamma_L,k_L)$ be models of $T$. 
\begin{enumerate}
\item One has ${\cal K}\equiv{\cal L}$ iff $k_K\equiv k_L$ (as difference fields) and $\Gamma_K\equiv\Gamma_L$ (as ordered difference groups).
\item Assume ${\cal K}\subseteq{\cal L}$. Then ${\cal K}\elex{\cal L}$ iff $k_K\elex k_L$ and $\Gamma_K\elex\Gamma_L$.
\end{enumerate}

\item Let $L/K$ be an immediate extension of valued difference fields, living in a model of $T$. Assume that $\ac(K)\subseteq k_K$, and let $a$ be a tuple from $L$. Then $\qftp(a/K) \vdash \tp(a/K)$.\end{enumerate}
\end{lemma}

\begin{proof}
(1) is clear by inspection of the language. 

To prove (2), note that for 
every $b\in L$ there is $c\in K$ and $b'\in\O_L^{\times}$ such that $b=cb'$. But 
then $\ac(b)=\ac(c)\res(b')$, showing that $\ac(b)\in k_L=k_K$. The result follows.
\end{proof}

Before we treat various cases of $\sigma$-henselian valued difference fields, let us mention 
a classical result of Pas. 

\begin{fact}[{\cite{Pas89}}]\label{F:Pas-Thm}
Let $T$ be the theory of henselian $\ac$-valued fields of residue characteristic 0, in the language of Pas $\L_{\rf,\vg}\cup\{\ac\}$. 
Then $T$ eliminates $\vf$-quantifiers.
\end{fact}

\subsection{$\sigma$-henselianity in contractive valued difference fields}\label{Sub:Hensel-Contractive}

\begin{definition}
A valued difference field $\mathcal{K}=(K,\Gamma_K,k_K,\sigma)$ is called \emph{contractive }if its value group $\Gamma_K$ is an $\omega$-increasing ordered difference group.
\end{definition}

\smallskip

Let ${\cal K}=(K,\Gamma_K,k_K,\sigma)$ be a contractive valued difference field, with \emph{fixed field}
$F:=\Fix(\sigma):=\{a\in K\mid \sigma(a)=a\}$. Then $\res\upharpoonright_F$ is injective. In particular, $\car(K)=\car(k_K)$. (In case ${\cal K}$ is $\sigma$-henselian in the sense of the definition below, $\res$ induces 
an isomorphism between $F$ and $\Fix(\overline{\sigma})$.)

For $g(X)\in K[X]_{\sigma}$ non-constant, define $\order(g)$ to be the 
minimal $n$ such that $g$ may be written as $g(X)=G(X,\sigma(X),\sigma^2(X),\ldots,\sigma^n(X))$, for some polynomial $G\in K[X_0,\ldots,X_n]$. If $\order(g)=n$, 
we put 
$$\complexity(g):=(n,\deg_{X_n}(G),\deg(G))\in\N^3,$$ where $\deg(G)$ is the total degree of $G$. We say that $g$ has \emph{smaller complexity }than $h$ if 
$\complexity(g)<_{lex}\complexity(h)$.

\smallskip

Recall that for any $G\in K[\overline{X}]$ there are (unique) polynomials $G_{\mu}\in K[\overline{X}]$ 
such that $G(\overline{Y}+\overline{X})=\sum_{\mu}G_{\mu}(\overline{Y})\overline{X}^{\mu}$. Here, 
$\overline{X}=(X_0,\ldots,X_n)$, $\mu=(\mu_0,\ldots,\mu_n)\in\N^{n+1}$ is a multi-index, and 
$\overline{X}^\mu:=\prod_{i=0}^n X_i^{\mu_i}$. 

From this, for $g(X)=G(X,\sigma(X),\sigma^2(X),\ldots,\sigma^n(X))\in K[X]_{\sigma}$ as above, we 
get the following Taylor expansion of difference polynomials (in one variable) 
$$g(a+X)=\sum_{\mu}g_{\mu}(a)X^{\mu},$$
where $g_{\mu}(X)=G_{\mu}(X,\sigma(X),\ldots,\sigma^n(X))$ and $X^\mu :=\prod_{i=0}^n (\sigma^{i}(X))^{\mu_i}$ for every multi-index $\mu=(\mu_0,\ldots,\mu_n)$.

\smallskip

Let us introduce some notation which will be used in the following definition. For $\mu\in\N^{n+1}$ 
and $\gamma\in\Gamma$, we let $|\mu|:=\sum\mu_i=1$, and $\sigma^\mu(\gamma):=\sum_{i=0}^n\mu_i\sigma^{i}(\gamma)$.

\begin{definition}[{\cite[Definition 4.5]{Azg10}}]\label{D:sigma-H}\mbox{}
\begin{itemize}
\item Let $K$ be contractive, $g(X)\in K[X]_{\sigma}$, $\order(g)\leq n$, let $a\in K$.

We say that $(g,a)$ is in \emph{$\sigma$-Hensel configuration} if $g\not\in K$ and if there exists 
$\gamma\in\Gamma_K$ and a multi-index $\mu\in\N^{n+1}$ 
with $|\mu|=1$ such that the following holds:

\begin{enumerate}
\item[(i)] $\val(g(a))=\val(g_{\mu}(a))+\sigma^{\mu}(\gamma)\leq \val(g_{\nu}(a))+\sigma^{\nu}(\gamma)$ 
for all $\nu\in\N^{n+1}$ with $|\nu|=1$;
\item[(ii)] $\val(g_{\nu}(a))+\sigma^{\nu}(\gamma)< \val(g_{\nu+\rho}(a))+\sigma^{\nu+\rho}(\gamma)$ 
for all non-zero $\nu,\rho\in\N^{n+1}$ with $g_{\nu}\neq0$.
\end{enumerate} 
\smallskip
Put $\gamma(g,a):=$ the $\gamma$ from above (this is uniquely determined \cite{Azg10}).
\smallskip

\item A contractive valued difference field $K$ is called \emph{$\sigma$-henselian} if for every $(g,a)$ in $\sigma$-Hensel 
configuration there exists $b\in K$ with $\val(b-a)=\gamma(g,a)$ and $g(b)=0$.
\end{itemize}
\end{definition}

\begin{remark}\label{R:LinDiffCl}
\begin{enumerate}
\item Let $(K,k,\Gamma,\sigma)$ be contractive and $\sigma$-henselian. Then $(k,\bar{\sigma})$ is \emph{linearly difference closed}, i.e.\ for 
all $\alpha_0,\ldots,\alpha_n\in k$ not all 0, the equation $1+\alpha_0X+\alpha_1\sigma(X)+\cdots+\alpha_n\sigma^n(X)$ has 
a solution in $(k,\bar{\sigma})$ \cite[Lemma 4.6]{Azg10}.

\item Conversely, let $(K,k,\Gamma,\sigma)$ be contractive with $(k,\bar{\sigma})$ a linearly closed difference 
field of characteristic 0. Assume that $K$ is a maximally complete valued field. Then $(K,k,\Gamma,\sigma)$ is $\sigma$-henselian by \cite[Corollary 5.6]{Azg10}. 

\item We now isolate a special case of (2). Let $(\Gamma,\sigma)$ be an $\omega$-increasing ordered difference group and $(k,\bar{\sigma})$ a linearly closed difference 
field of characteristic 0. Then, the Hahn field $K:=k((\Gamma))$, 
a valued field which is naturally equipped with an automorphism $\sigma$, namely $\sigma(\sum_{\gamma}a_{\gamma}t^{\gamma}):=\sum_{\gamma}\overline{\sigma}(a_{\gamma})t^{\sigma(\gamma)}$, is a contractive $\sigma$-henselian valued difference field. \end{enumerate}
\end{remark}

\smallskip

For completeness, let us mention another consequence of the results in \cite{Azg10}. (We will make no use of it in our paper.) 

\begin{remark}
A $\sigma$-henselian contractive valued difference field of characteristic 0 is henselian (as a valued field).
\end{remark}

Indeed, combining Remark \ref{R:LinDiffCl}(1\&3) with Fact \ref{F:sigmaAKE}(4) below, one sees 
that every $\sigma$-henselian contractive valued difference field (of characteristic 0) is elementarily equivalent to a valued 
difference field with underlying valued field a Hahn field (which is henselian).

\smallskip

For $q=p^n$ a prime power, consider ${\cal K}_q=(K_q,\Gamma,k,\phi_q)$, where $(K_q,\Gamma,k)\models \ACVF_{p,p}$ and $\phi_q$ 
is the Frobenius automorphism $x\mapsto x^q$.

The following follows in a straight forward way from \cite[Section 8]{Azg10} (see \cite[Proof of Thm 4.3.24]{Gia11} for a proof).

\begin{fact}\label{F:Frob-Hensel}
For every non-principal ultrafilter $\cal{U}$ on the set $Q$ of prime powers, $\prod_{{\cal U}}{\cal K}_q$ is $\sigma$-henselian, i.e.\ the non-standard Frobenius automorphism 
is $\sigma$-henselian.
\end{fact}

\subsection{AKE principle in the contractive case and $\VFA$}\label{Sub:AKE}

Denote by $T_0$ the theory of $\sigma$-henselian contractive valued difference fields of equal characteristic 0 (in the language $\L_{\rf,\vg,\sigma}$), and by 
$T_0^{\ac}$ that of $\sigma$-henselian contractive $\ac$-valued difference fields (in the language $\L_{\rf,\vg,\sigma}\cup\{\ac\}$).

\begin{fact}[Durhan]\label{F:sigmaAKE}\mbox{}
\begin{enumerate}
\item The theory $T_0^{\ac}$ eliminates quantifiers from the field sort $\vf$.
\item In any model ${\cal K}=(K,\Gamma_K,k_K)\models T_0^{\ac}$, $k_K=\rf({\cal K})$ is stably embedded, and the induced structure is that of a difference field. Similarly, $\Gamma_K=\vg({\cal K})$ is 
stably embedded and is a pure ordered $\Z[\sigma]$-module. Moreover, $\rf$ and $\vg$ are orthogonal.
\item Let ${\cal K}=(K,\Gamma_K,k_K)$ and ${\cal L}=(L,\Gamma_L,k_L)$ be models of $T_0^{\ac}$. 
\begin{enumerate}
\item ${\cal K}\equiv{\cal L}$ iff $k_K\equiv k_L$ (as difference fields) and $\Gamma_K\equiv\Gamma_L$ (as ordered $\Z[\sigma]$-modules).
\item Suppose ${\cal K}\subseteq{\cal L}$. Then ${\cal K}\elex{\cal L}$ iff $k_K\elex k_L$ and $\Gamma_K\elex\Gamma_L$.
\end{enumerate}

\item Let ${\cal K}=(K,\Gamma_K,k_K)$ and ${\cal L}=(L,\Gamma_L,k_L)$ be models of $T_0$. 
\begin{enumerate}
\item ${\cal K}\equiv{\cal L}$ if and only if $k_K\equiv k_L$ and $\Gamma_K\equiv\Gamma_L$.
\item Suppose ${\cal K}\subseteq{\cal L}$. Then ${\cal K}\elex{\cal L}$ if and only if $k_K\elex k_L$ and $\Gamma_K\elex\Gamma_L$.
\end{enumerate}
\end{enumerate}
\end{fact}

\begin{proof}
(1) is \cite[Thm 4.5.2]{Azg07}, and (2) and (3) follow from (1) (by Lemma \ref{L:QE-consequences}). (4a) is \cite[Thm 4.5.1]{Azg07}. Finally, to prove 
the non-trivial implication in (4b), taking an elementary extension of the pair $(K,L)$, we may assume that 
both $K$ and $L$ are $\aleph_1$-saturated. In \cite[Proof of Thm 11.6]{Pal10}, it is shown that 
$\O^{\times}_M$ is a pure $\Z[\sigma]$-submodule of $M^{\times}$ for every contractive valued difference field $M$. The saturation assumption implies that $K^{\times}$, $L^{\times}$ and $\Gamma_K$ are all pure-injective $\Z[\sigma]$-modules (see \cite[Section 10.7]{Hod93} for 
facts about pure-injective modules). It follows that we get splittings for the exact sequence 
(of $\Z[\sigma]$-modules)
$1\rightarrow\O^{\times}_K\rightarrow K^{\times}\rightarrow\Gamma_K\rightarrow 0$ and 
similarly for $1\rightarrow\O^{\times}_L\rightarrow L^{\times}\rightarrow\Gamma_L\rightarrow 0$, i.e.\ cross-sections $s_K:\Gamma_K\rightarrow K^{\times}$ 
and $s_L:\Gamma_L\rightarrow L^{\times}$ commuting with $\sigma$. 

Moreover, as $\Gamma_K\elex\Gamma_L$, in particular 
$\Gamma_K$ is a pure submodule of $\Gamma_L$, and so there is a $\Z[\sigma]$-submodule 
$\Delta\leq\Gamma_L$ such that $\Gamma_L=\Gamma_K\oplus\Delta$. Now let 
$s_L'=s_K\oplus s_L\upharpoonright_\Delta:\Gamma_L=\Gamma_K\oplus\Delta\rightarrow L^{\times}$. 
Then $s_L'$ is a cross-section (commuting with $\sigma$) which extends $s_K$.

As any cross-section $s$ gives rise to an $\ac$-map, letting 
$\ac(x)=\res(x\cdot s(\val(x))^{-1})$, this shows that we may expand $K$ and $L$ 
so that the embedding is an embedding of models of $T_0^{\ac}$. We then conclude by part (3a). 
\end{proof}

\begin{lemma}\label{L:ACFALinCl}
Any model of $\ACFA$ is linearly difference closed.
\end{lemma}
\begin{proof}
Let $(k,\sigma)$ be a difference field, and let $\lambda(X)=1+\alpha_0X+\alpha_1\sigma(X)+\cdots+\alpha_n\sigma^n(X)$, where all $\alpha_i$ are from $k$ and $\alpha_n\neq0$. 

Applying a suitable power of $\sigma^{-1}$ to the equation, we may assume that $\alpha_0\neq0$. If $n=0$, the statement is trivial. 

Now assume $n>0$. Consider $l:=k(X_0,\ldots,X_{n-1})$, and extend $\sigma$ to $l$, letting 
$\sigma(X_i):=X_{i+1}$ for $i<n-1$ and 
$\sigma(X_{n-1}):=-\frac{1}{\alpha_n}(1+\sum_{i=0}^{n-1}\alpha_iX_i)$. Then $\lambda(X_0)=0$, proving that existentially closed difference fields are linearly difference closed.
\end{proof}

Let $\IncVFA$ be the theory of structures ${\cal K}=(K,k,\Gamma)$ (in the language $\L_{\rf,\vg,\sigma}\cup\{\ac\}$), where $(K,k_K,\Gamma_K)$ is a contractive 
$\ac$-valued difference field of characteristic 0, $(k,\overline{\sigma})$ is a difference field 
containing $k_K$, and $(\Gamma,\sigma_{\Gamma})\models\IncODG$ contains $\Gamma_K$ as a difference subgroup. (I.e. we do not require that the maps $\res$ and 
$\val$ are surjective.) 

Let $\VFA$ be the theory $T_0^{\ac}$, together with axioms expressing that $(\rf,\overline{\sigma})\models\ACFA_0$ and that $(\vg,\sigma)\models\IncDODG$. 

\begin{fact}\label{F:ModComp}
$\VFA$ is the model-companion of $\IncVFA$, in the language of $\ac$-valued difference fields. The same result holds if both $\VFA$ and $\IncVFA$ are restricted to the language of valued difference fields.
\end{fact}

\begin{proof}
For valued difference fields (without $\ac$-map), one may show, using 
the proofs of \cite[2.5 \& 2.6]{AzVa11} and \cite[3.3 -- 3.5]{Azg10}, that every model of $\IncVFA$ embeds into some ${\cal K}=(K,k,\Gamma)\models\IncVFA$ with surjective $\res$ and $\val$ and such that $(k,\overline{\sigma})\models\ACFA_0$ and $(\Gamma,\sigma_\Gamma)\models\IncDODG$. (We 
omit the details.) The maximal immediate extension of 
${\cal K}$ is then $\sigma$-henselian, by Remark \ref{R:LinDiffCl}(2) and Lemma \ref{L:ACFALinCl}, so it is a model of $\VFA$. As $\ACFA_0$ and $\IncDODG$ are model-complete, model-completeness 
of $\VFA$ follows from Fact \ref{F:sigmaAKE}(4b). (See also \cite[Thm 4.3.20]{Gia11}.) 

To finish the proof, it is enough to show that $\VFA$ is a companion of $\IncVFA$ 
(in the language with $\ac$). This means that we need to show that every model of $\IncVFA$ embeds 
into a model of $\VFA$. 

Let ${\cal K}=(K,k,\Gamma)\models\IncVFA$ be given. We first show the special case where $\Gamma=\Gamma_K$. Using only the constructions from 
\cite[Proofs of 2.5 \& 2.6]{AzVa11}, the value group does not increase, and so, similarly to the case without $\ac$, we may embed ${\cal K}$ into some 
${\cal K'}=(K',k',\Gamma)\models T_0$, with $(k',\overline{\sigma})\models\ACFA_0$. Moreover, since the value group is the same, $\ac$ uniquely extends to ${\cal K'}$ making it a model of $T_0^{\ac}$. 
By Fact \ref{F:sigmaAKE}, ${\cal K'}\equiv(k'((\Gamma),k',\Gamma)={\cal L}'$, where on ${\cal L}'$ we take the standard $\ac$-map given by the first non-zero coefficient. 
Now ${\cal L'}$ embeds into a model of $\VFA$, namely into $(k'((\tilde{\Gamma})),k',\tilde{\Gamma})$, where $\Gamma\subseteq\tilde{\Gamma}\models\IncDODG$. 
Since ``admitting an embedding into a model of $\VFA$'' is a first order property, the proof of the special case is finished.

For the general case, let ${\cal K}=(K,k,\Gamma)\models\IncVFA$ be given. By the special case, $(K,k,\Gamma_K)\subseteq{\cal L}=(L,l,\Delta)\models\VFA$. 
We may choose ${\cal L}$ sufficiently saturated. Then $\Gamma$ embeds into $\Delta$ over $\Gamma_K$, and we may conclude.
\end{proof}

With a slightly different notion of $\sigma$-henselianity, Fact \ref{F:ModComp} had been independently obtained by Hrushovski in \cite{Hru02}, where the following consequence  is also mentioned. (See also \cite[Thm 4.3.24]{Gia11}.)

\begin{fact}[Hrushovski]\label{F:nsf} Let $\phi$ be a sentence in the language of $\ac$-valued difference fields. The following are equivalent:
\begin{enumerate}
\item $\VFA\vdash\phi$;
\item ${\cal K}_p\models\phi$ for all large enough primes $p$.
\end{enumerate}
\end{fact}

\begin{proof}
Every instance of the $\sigma$-Hensel scheme holds in ${\cal K}_p$ for $p\gg0$, by Fact \ref{F:Frob-Hensel}. Moreover, it is easy to see that every axiom of 
ordered $\Q(\sigma)$-vector spaces holds in $(\Gamma_{K_p},\gamma\mapsto p\gamma)$ for $p\gg0$. By Fact \ref{F:sigmaAKE}, it is thus enough to show that the limit 
theory of the residue difference fields coincides with $\ACFA_0$. This is true, by a very deep result of Hrushovski \cite{Hru04}.
\end{proof}

\begin{examples}
\begin{enumerate}
\item Let ${\cal U}$ be a non-principal ultrafilter on the set of prime numbers. Then $\prod_{\cal U}{\cal K}_p\models \VFA$.
\item Let $(k,\overline{\sigma})\models\ACFA_0$ (e.g.\ $k=\C$ and $\overline{\sigma}$ a sufficiently 'generic' automorphism of $\C$), and let $\Gamma$ be a non-trivial ordered 
vector space over $\Q(\sigma)$. Then $K:=k((\Gamma))\models\VFA$.
\end{enumerate}
\end{examples}

\begin{proof}
(1) follows from Fact \ref{F:nsf}, and (2) is a consequence of Lemma \ref{L:ACFALinCl} together with Remark \ref{R:LinDiffCl}(3).
\end{proof}

\subsection{Isometric valued difference fields}
Another important class of valued difference fields is the class of valued fields with an \emph{isometry}, where one requires that the induced automorphism $\sigma_{\Gamma}$ on the value 
group $\Gamma$ is the identity. 
The model theory of $\sigma$-henselian valued fields of residue characteristic 0 with an isometry is well understood, if one assumes in addition that there are \emph{enough constants}, i.e.\ 
that every $\gamma\in\Gamma$ is of the form $\val(a)$ for some $a\in\Fix(\sigma)$. 

By work of Scanlon \cite{Sca00,Sca03}, B\'elair, Macintyre and Scanlon \cite{BeMaSc07} and 
then (in a slightly more general setting) Durhan and van den Dries \cite{AzVa11}, in this context one may eliminate quantifiers from the field sort $\vf$ in the language with angular components and thus get an 
Ax-Kochen-Ershov principle, analogous to Fact \ref{F:sigmaAKE}. (In Fact \ref{F:IsomAKE} below we give a precise statement. For the definition of $\sigma$-henselianity 
in the isometric case, we refer to \cite[Definition 4.4]{AzVa11}.) Moreover, a model-companion exists in the isometric case (see e.g., \cite{BeMaSc07}).

Let $S_0^{ac}$ be the theory of $\sigma$-henselian valued fields with an isometry (having enough constants) in residue characteristic 0, considered in the language $\L_{\rf,\vg,\sigma}\cup\{\ac\}$ of $ac$-valued 
difference fields.

\begin{fact}[\cite{AzVa11}]\label{F:IsomAKE}
\begin{enumerate}
\item The theory $S_0^{\ac}$ eliminates $\vf$-quantifiers.
\item In any model ${\cal K}=(K,\Gamma_K,k_K)\models S_0^{\ac}$, $k_K=\rf({\cal K})$ is stably embedded, and the induced structure is that of a difference field. Similarly, $\Gamma_K=\vg({\cal K})$ is 
stably embedded and is a pure ordered abelian group. Moreover, $\rf$ and $\vg$ are orthogonal.
\item Let ${\cal K}=(K,\Gamma_K,k_K)$ and ${\cal L}=(L,\Gamma_L,k_L)$ be models of $S_0^{\ac}$. 
\begin{enumerate}
\item ${\cal K}\equiv{\cal L}$ if and only if $k_K\equiv k_L$ (as difference fields) and $\Gamma_K\equiv\Gamma_L$ (as ordered abelian groups).
\item Suppose ${\cal K}\subseteq{\cal L}$. Then  ${\cal K}\elex{\cal L}$ if and only if $k_K\elex k_L$ and $\Gamma_K\elex\Gamma_L$.
\end{enumerate}
\end{enumerate}
\end{fact}

For $p$ a prime number, let $W(\F_p^{alg})$ be the quotient field of the ring of Witt vectors 
with coefficients from $\F_p^{alg}$, with its natural valuation. On the valued field $W(\F_p^{alg})$, there 
is a natural isometry, namely the \emph{Witt-Frobenius} automorphism which we denote 
by $\widetilde{Frob}_p$, sending $x=\sum_n a_n p^n\in W(\F_p^{alg})$ to 
$\sum_n a_n^p p^n$. Letting $\ac(x):=a_{\val(x)}$, we get an $\ac$-valued difference 
field $\mathcal{W}_p=(W(\F_p^{alg}),\Z,\F_p^{alg},\widetilde{Frob}_p)$.

The following example is discussed in \cite[Section 12]{BeMaSc07}.

\begin{example}[Non-standard Witt-Frobenius automorphism]\label{E:Witt-Frob}
Let $\mathcal{U}$ be a non-principal ultrafilter on the set of prime numbers. 
Then $\prod_{\mathcal{U}}\mathcal{W}_p\models S_0^{ac}$. Moreover, 
$\prod_{\mathcal{U}}\mathcal{W}_p\equiv\prod_{\mathcal{U}}(\F_p^{alg}((t)),\sigma_p)$, where 
$\sigma_p$ is the isometry given by $\sum a_n t^n\mapsto \sum a_n^p t^n$.
\end{example}

\begin{remark}\label{R:Mult-Case}
In the setting of so-called \emph{multiplicative} valued difference fields, forming a common generalisation of the contractive and the isometric case, 
Pal established similar Ax-Kochen-Ershov type results (see \cite{Pal10}), even without adding an angular component map and working in 
the appropriate language for the $\RV$ sort, where $\RV=K^{\times}/1+\m$.
\end{remark}

\section{Indiscernible arrays and {$\NTP_2$}} \label{S:NTP2}

In this section we recall some facts about $\NTP_{2}$ and prove some
new lemmas. As usual, we fix a monster model $\M$. We don't distinguish here between finite and infinite tuples unless mentioned explicitly, and $\bar{a}$, $\bar{b}$, ... denote infinite sequences. 

\begin{definition} \label{D:NTP2}
We say that $\varphi\left(x,y\right)$ has $\TP_{2}$ if there are
$\left(a_{ij}\right)_{i,j\in\omega}$ and $k\in\omega$ such that:
\begin{enumerate}
\item $\left\{ \varphi\left(x,a_{ij}\right)\right\} _{j\in\omega}$ is $k$-inconsistent
for every $i\in\omega$.
\item $\left\{ \varphi\left(x,a_{if\left(i\right)}\right)\right\} _{i\in\omega}$
is consistent for every $f:\,\omega\to\omega$.
\end{enumerate}
A theory is called $\NTP_{2}$ if no formula has $\TP_{2}$.
\end{definition}

\begin{fact} \label{F:NIPSimpleIsNTP2}
If a theory $T$ is simple or $\NIP$, then it is $\NTP_{2}$ (see e.g. \cite[Section 2]{Che12}).\end{fact}

\begin{fact}[\cite{Che12}]\label{F:1Variable} If $T$ is not $\NTP_{2}$, then already some $\varphi\left(x,y\right)$
with $\left|x\right|=1$ has $\TP_{2}$.\end{fact}
\begin{definition}
We say that $\left(c_{ij}\right)_{i,j\in\kappa}$ is a \emph{strongly indiscernible array} if $\bar{c}_{i}=(c_{ij})_{j\in\kappa}$ is indiscernible
over $\bar{c}_{\neq i}$ for all $i$ and $\left(\bar{c}_{i}\right)_{i\in\kappa}$
is an indiscernible sequence (of sequences).
\end{definition}
~

We start with some auxiliary results on finding indiscernible sequences and arrays.
\begin{lemma}
\label{lem: auxiliary lemmas on indisc} 
\begin{enumerate}
\item \label{lem_item: consistentcy of sequence implies can move} Let $C$ be a small set, $\bar{a}=\left(a_{i}\right)_{i\in\omega}$
be a $C$-indiscernible sequence, $b$ given, and let $p\left(x,a_{0}\right)=\tp\left(b/a_{0}C\right)$.
Assume that $\bigcup_{i\in\omega}p\left(x,a_{i}\right)$ is consistent.
Then there is some $\bar{a}'$ indiscernible over $bC$ and such that
$\bar{a}'\equiv_{a_{0}C}\bar{a}$.

\item \label{lem_item: ramsey for arrays} Let $C$ be a small set and $\left(a_{\alpha i}\right)_{\alpha<n,i<\omega}$
be an array with $n<\omega$. Then for any finite $\Delta\in L(C)$
and $N<\omega$ we can find $\Delta$-mutually indiscernible sequences
$\left(a_{\alpha,i_{\alpha0}},...,a_{\alpha,i_{\alpha N}}\right)\subset\bar{a}_{\alpha}$,
$i_{\alpha0}<\ldots<i_{\alpha N}\in\omega$, $\alpha<n$.

\item \label{lem_item: almost indisc gives indisc} Assume that we are given $\left(\bar{a}_{i}\right)_{i\in\kappa}$
and a small set $C$ such that $\bar{a}_{i}$ is indiscernible over
$\bar{a}_{<i}\left(a_{j0}\right)_{j>i}C$ for all $i\in\kappa$. Then
there exists an array $\left(\bar{a}_{i}'\right)_{i\in\kappa}$ such
that $\bar{a}_{i}'\equiv_{a_{i0}C}\bar{a}_{i}$ and $\bar{a}'_{i}$
is indiscernible over $\bar{a}'_{\neq i}C$ for all $i$.
\end{enumerate}
\end{lemma}
\begin{proof}
(\ref{lem_item: consistentcy of sequence implies can move}) By applying an automorphism it is enough to find $b'\equiv_{a_{0}C}b$ such
that $\bar{a}$ is indiscernible over $b'C$. Let $\Delta$ be an
arbitrary finite set of formulas over $C$. Let $b^{+}\models\bigcup_{i\in\omega}p\left(x,a_{i}\right)$.
By Ramsey there is an infinite subsequence $\bar{a}^{+}$ of $\bar{a}$
which is $\Delta$-indiscernible over $b^{+}$. Let $\sigma$ be a
$C$-automorphism sending $\bar{a}^{+}$ to $\bar{a}$. Then $\bar{a}$
is $\Delta$-indiscernible over $\sigma\left(b^{+}\right)$ and $\sigma\left(b^{+}\right)\equiv_{a_{0}C}b$.
We find $b'$ by compactness.

(\ref{lem_item: ramsey for arrays}) By the finitary Ramsey theorem there are natural numbers $\left(N_{\alpha}\right)_{\alpha<n}$
such that for every $\alpha < n$ and every set $A$ of size $\sum_{\beta<\alpha}N_{\alpha}+\left(n-1 -\alpha\right)\times N$,
every sequence of elements of length $N_{\alpha}$ contains a subsequence
of length $N$ which is $\Delta$-indiscernible over $A$.

Let $\bar{a}_{\alpha}^{+}=\left(a_{\alpha i}\right)_{i<N_{\alpha}}$.
By reverse induction and the choice of $N_{\alpha}$ we can find $\bar{a}_{\alpha}'$
such that:
\begin{itemize}
\item $\bar{a}_{\alpha}'$ is a subsequence of $\bar{a}_{\alpha}^{+}$,
\item $\left|\bar{a}_{\alpha}'\right|=N$,
\item $\bar{a}_{\alpha}'$ is $\Delta$-indiscernible over $\bar{a}_{<\alpha}^{+}\bar{a}_{>\alpha}'$.
\end{itemize}
But then $\bar{a}_{0}',\ldots,\bar{a}_{n-1}'$ are as wanted.

(\ref{lem_item: almost indisc gives indisc}) By compactness, it is enough to prove the statement for finite $\kappa$. Let $\Delta\in L\left(C\right)$
finite and $N\in\omega$ be arbitrary. By (\ref{lem_item: ramsey for arrays}) we can find $\Delta$-mutually
indiscernible sequences $\bar{a}_{\alpha}^{+}=\left(a_{\alpha,i_{\alpha0}},...,a_{\alpha,i_{\alpha N}}\right)\subset\bar{a}_{\alpha}$
for $\alpha\in\kappa$. It follows from the assumption that $a_{0,i_{00}}a_{1,i_{10}}\ldots a_{\kappa-1,i_{\left(\kappa-1\right)0}}\equiv_{C}a_{00}a_{10}\ldots a_{\left(\kappa-1\right)0}$. 
Let $\sigma$ be a $C$-automorphism sending the former to the latter.
Then we have:
\begin{itemize}
\item $\sigma\left(\bar{a}_{0}^{+}\right),\ldots,\sigma\left(\bar{a}_{\kappa-1}^{+}\right)$
are mutually $\Delta$-indiscernible,
\item $a_{\alpha,i_{\alpha0}},...,a_{\alpha,i_{\alpha N}}\equiv_{C}a_{\alpha0}\ldots a_{\alpha N}$
by indiscernibility, so $\sigma\left(\bar{a}_{\alpha}^{+}\right)\equiv_{C}\left(a_{\alpha i}\right)_{i\leq N}$,
which together with $\sigma\left(a_{\alpha,i_{\alpha0}}\right)=a_{\alpha0}$
implies that $\sigma\left(\bar{a}_{\alpha}^{+}\right)\equiv_{a_{\alpha0}C}\left(a_{\alpha i}\right)_{i\leq N}$,
for each $\alpha$. 
\end{itemize}
By compactness we find $\bar{a}_{0}',\ldots,\bar{a}_{\kappa-1}'$
as wanted.
\end{proof}

\begin{lemma}
\label{lem: adding carelessly} Let $\left(c_{ij}\right)_{i,j\in\omega}$
be a strongly indiscernible array such that $\left(\bar{c}_{i}\right)_{i\in\omega}$
is indiscernible over $a$, and let $b$ be given. Then we can find $b_{ij}$
such that $\left(b_{ij}c_{ij}\right)_{i,j\in\omega}$ is a strongly indiscernible
array, $bc_{00}\equiv b_{ij}c_{ij}$ for all $i,j$ and $\left(\bar{b}_{i}\bar{c}_{i}\right)_{i\in\omega}$
is indiscernible over $a$.\end{lemma}
\begin{proof}
Set $b_{00}=b$ and let $b_{ij}$ be such that $b_{ij}c_{ij}\equiv b_{00}c_{00}$.
By Lemma \ref{lem: auxiliary lemmas on indisc}(\ref{lem_item: ramsey for arrays}) and compactness, applying an automorphism 
we may assume that $\left(\bar{b}_{i}\bar{c}_{i}\right)_{i\in\omega}$
are mutually indiscernible. By Ramsey, compactness and applying automorphisms over $a$, we may assume in addition that $\left(\bar{b}_{i}\bar{c}_{i}\right)_{i\in\omega}$
is indiscernible over $a$. 
\end{proof}

Given a definable set $D$, by $D_{\indu}$ we mean the full induced structure on it.
The next lemma is a generalisation of a lemma from \cite[Section 1]{Che12}. 
\begin{lemma}
\label{lem: finding indiscernible row} Let an  $\emptyset$-definable set $D$ be
stably embedded and assume that $D_{\indu}$ is $\NTP_{2}$. Let $\bar{b}\subset D$
with $|\bar{b}|\leq\lambda$ be given.

Assume that $\left(\bar{c}_{i}\right)_{i\in\kappa}$ is an array with
mutually indiscernible rows over $C$, and $\bar c_i = (c_{ij})_{j \in \omega}$. If $\kappa\geq\left(\lambda+|T|\right)^{+}$,
then there is $i\in\kappa$ and $\bar{c}'$ such that:
\begin{itemize}
\item $\bar{c}'\equiv_{c_{i0}C}\bar{c}_{i}$
\item $\bar{c}'$ is indiscernible over $C\bar{b}$.
\end{itemize}
\end{lemma}
\begin{proof}
Let $p_{i}(\bar{x},c_{i0})=\tp(\bar{b}/c_{i0}C)$. 

We claim that $\bigcup_{j\in\omega}p_i(\bar{x},c_{ij})$ is consistent
for some $i\in\kappa$.

Assume not, then by compactness and indiscernibility, for every $i\in\kappa$
we have some $\phi_{i}(x_{i},c_{i0}d_{i})\in p_{i}$ (with finite
$x_{i}\subset\bar{x}$), $d_{i}\in C$, and $k_{i}\in\omega$ such
that $\left\{ \phi_{i}(x_{i},c_{ij}d_{i})\right\} _{j\in\omega}$
is $k_{i}$-inconsistent. As $D$ is stably embedded, for each $i$ there is some $\psi_i(x_i,e_{i0})$
with $e_{i0}\in D$ such that $\psi_i(x_i,e_{i0})\cap D=\phi_i(x_i,c_{i0}d_{i})\cap D$.
As the type of $c_{i0}d_i$ says that there is an element $e_{i0}$ with this property, by the indiscernibility of the rows over $C$ we can find $e_{ij}\in D$ such that 
$\psi_i(x_i,e_{ij})\cap D=\phi_i(x_i,c_{ij}d_{i})\cap D$ for all $i,j$.
As $\kappa$ was chosen large enough, by throwing some rows away we may assume that $\psi_i = \psi$, $x_{i}=x$ and $k_{i}=k$.

But then we have:

\begin{itemize}
\item $\left\{ \psi(x,e_{ij})\land D(x)\right\} _{j\in\omega}$ is $k$-inconsistent
for every $i$ (as $\{ \phi_i(x_i,c_{ij}d_j) \}_{j\in \omega}$ is $k$-inconsistent),
\item $\left\{ \psi(x,e_{if(i)})\land D(x)\right\} _{i\in\kappa}$ is consistent for every $f: \kappa \to \omega$
(it is witnessed by $\bar{b}$ for $f(i)=0$, and follows for an arbitrary $f$ by the mutual indiscernibility of the rows).
\end{itemize}

This is a contradiction to $D_{\indu}$ being $\NTP_{2}$.

So let $i$ be as given by the claim. But then by Lemma \ref{lem: auxiliary lemmas on indisc}(\ref{lem_item: ramsey for arrays})
we can find $\bar{c}'\equiv_{c_{i0}C}\bar{c}_{i}$ such that $\bar{c}'$
is indiscernible over $\bar{b}C$.\end{proof}

\begin{lemma}[The Array Extension Lemma]
\label{prop: adding carefully} Let $D$ be a stably embedded $\emptyset$-definable set and assume that $D_{\indu}$
is $\NTP_{2}$. 

Assume that 
\begin{itemize}
\item $\left(\bar{c}_{i}\right)_{i\in\omega}$ is indiscernible over $a$,
\item $\left(c_{ij}\right)_{i,j\in\omega}$ is a strongly indiscernible array.
\end{itemize}

Let a small $b\subseteq D$ be given. Then we can find $\left(c_{ij}^{*}\right)_{i,j\in\omega}$
and $\left(b_{ij}^{*}\right)_{i,j\in\omega}$ such that:
\begin{enumerate}
\item $\left(\bar{b}_{i}^{*}\bar{c}_{i}^{*}\right)_{i\in\omega}$ is indiscernible
over $a$,
\item $\left(\bar{b}_{i}^{*}\bar{c}_{i}^{*}\right)_{i\in\omega}$ are mutually
indiscernible,
\item $\bar{c}_{i}^{*}\equiv_{c_{i0}}\bar{c}_{i}$ for all $i\in\omega$ (so in particular $c_{i0}^* = c_{i0}$),
\item $abc_{00}\equiv ab_{00}^{*}c_{00}^{*}$.
\end{enumerate}

In particular $\left(b_{ij}^{*}c_{ij}^{*}\right)_{i,j\in\omega}$
is a strongly indiscernible array.\end{lemma}

\begin{proof}
By compactness, it suffices to prove the result for finite $b$.

First, by indiscernibility, Ramsey and applying automorphisms over $a$, we
can find $b_{i}$ such that $ab_{i}\bar{c}_{i}\equiv ab\bar{c}_{0}$
and $\left(b_{i}\bar{c}_{i}\right)_{i\in\omega}$ is indiscernible
over $a$.

Again by compactness, indiscernibility and applying automorphisms over $a$,
it is enough to find $\bar{c}_{< n}^{*},\bar{b}_{< n}^{*}$
satisfying (2), (3) and (4$'$) for every $n\in\omega$, where

(4$'$) $abc_{00} \equiv ab_{k0}^*c_{k0}^* $ for all $k<n$.

So fix $n\in\omega$ and let $I=I_{0}+I_{1}+\ldots+I_{n-1}=|T|^{+}+\ldots+\left|T\right|^{+n}$ (where for a cardinal $\kappa$ we let $\kappa^{+n}$ denote the $n$th successor of $\kappa$).
By compactness we may expand our sequence to $\left(b_{i}\bar{c}_{i}\right)_{i\in I}$
with the same Ehrenfeucht-Mostowski type over $a$. 

By reverse induction on $k < n$ we find $i_{k}$, $\bar{c}_{k}^{+}$,
$\bar{b}_{k}^{+}$ such that:
\begin{enumerate}
\item [{(a)}] $i_{k}\in I_{k}$,
\item [{(b)}] $\bar{c}_{k}^{+}\equiv_{c_{i_{k}0}}\bar{c}_{i_{k}}$ (so in particular $c_{k0}^+ = c_{i_k 0}$),
\item [{(c)}] $\bar{c}_{\in I_{<k}}\bar{c}_{k}^{+}\bar{c}_{k+1}^{+} \ldots \bar{c}_{n-1}^{+}\equiv\bar{c}_{\in I_{<k}}\bar{c}_{i_{k}}\bar{c}_{i_{k+1}}\ldots\bar{c}_{i_{n-1}}$,

\item [{(d)}] $b_{k0}^+ = b_{i_k}$ and $\bar{b}_{k}^{+}\subseteq D$,

\item [{(e)}] $\left(b_{kj}^{+}c_{kj}^{+}\right)_{j\in\omega}$ is indiscernible
over $\bar{b}_{>k}^{+}\bar{c}_{>k}^{+}b_{\in I_{<k}}\bar{c}_{\in I_{<k}}$.
\end{enumerate}
In step $k$, let $C=\bar{c}_{>k}^{+}\bar{c}_{\in I_{<k}}$ and $\bar{b}=\bar{b}_{>k}^{+}b_{\in I_{<k}}$.
Then $\bar{b}\subseteq D$, $\left|\bar{b}\right|\leq\left|T\right|^{+k}$
and $\left(\bar{c}_{i}\right)_{i\in I_{k}}$ are mutually indiscernible
over $C$ (by (c) for $k+1$ and the assumption on $\left(\bar{c}_{i}\right)_{i \in I}$). As $I_{k}=\left|T\right|^{+\left(k+1\right)}$,
it follows by Lemma \ref{lem: finding indiscernible row} that there
is some $i_{k}\in I_{k}$ and $\bar{c}_{k}'$ indiscernible over $\bar{b}C$
and such that $\bar{c}_{k}'\equiv_{c_{i_{k}0}C}\bar{c}_{i_{k}}$.
Let $b_{k0}'=b_{i_{k}}$ and $b_{kj}'$ be such that $b_{k0}'c_{k0}'\equiv_{\bar{b}C}b_{kj}'c_{kj}'$.
By Ramsey, compactness and $\bar{b}C$-automorphisms we can find a
$\bar{b}C$-indiscernible sequence $\left(b_{kj}^{+}c_{kj}^{+}\right)_{j\in\omega}$
such that $b_{k0}^{+}c_{k0}^{+}=b_{k0}'c_{k0}'$ and $\bar{c}_{k}^{+}\equiv_{\bar{b}C}\bar{c}_{k}'$.
Now (b) and (c) follow from (c) for $k+1$, and $\bar{c}_{k}^{+}\equiv_{C}\bar{c}_{k}'\equiv_{C}\bar{c}_{i_{k}}$
and $c_{k0}^{+}=c_{i_{k}0}$. Parts (a), (d) and (e) are clearly satisfied by construction.

By Lemma \ref{lem: auxiliary lemmas on indisc}(\ref{lem_item: almost indisc gives indisc}) and (e) we find
sequences $\left(b_{kj}^{++}c_{kj}^{++}\right)_{j\in\omega}$ for
$k<n$ which are mutually indiscernible and such that $\bar{c}_{k}^{++}\bar{b}_{k}^{++}\equiv_{b_{k0}^{+}c_{k0}^{+}}\bar{c}_{k}^{+}\bar{b}_{k}^{+}$.

~

Finally, let $\sigma$ be an $a$-automorphism sending $b_{i_{<n}}\bar{c}_{i_{<n}}$to
$b_{<n}\bar{c}_{<n}$, and let $\bar{b}_{k}^{*}=\sigma\left(\bar{b}_{k}^{++}\right)$
and $\bar{c}_{k}^{*}=\sigma\left(\bar{c}_{k}^{++}\right)$ for $k<n$.

We have:
\begin{itemize}
\item $\left(\bar{b}_{k}^{*}\bar{c}_{k}^{*}\right)_{k<n}$ are mutually
indiscernible (as $\left(\bar{b}_{k}^{++}\bar{c}_{k}^{++}\right)_{k<n}$
are),
\item $\bar{c}_{k}^{*}\equiv_{c_{k0}}\bar{c}_{k}$ (as $\bar{c}_{k}^{++}\equiv_{c_{k0}^{+}}\bar{c}_{k}^{+}$,
$\bar{c}_{k}^{+}\equiv_{c_{i_{k}0}}\bar{c}_{i_{k}}$, $c_{k0}^{+}=c_{i_{k}0}$
and $\sigma$ is an automorphism),
\item $abc_{00}\equiv ab_{k0}^{*}c_{k0}^{*}$ for all $k<n$, as  $b_{k0}^{*}c_{k0}^*=b_{k}c_{k0}$ by the construction and $ab_kc_{k0} \equiv abc_{00}$. 
\end{itemize}
\end{proof}

\begin{lemma}
\label{lem: very indisc witness of TP2} 
In any theory, $\phi(x,y)$ has $\TP_2$ if and only if there is a strongly indiscernible array $(a_{ij})_{i,j \in \omega}$ witnessing it (as in Definition \ref{D:NTP2}) and $c \models \{ \phi(x,a_{i0}) \}_{i \in \omega}$ such that the sequence of rows $(\bar{a}_i)_{i \in \omega}$ is indiscernible over $c$.
\end{lemma}
\begin{proof}
By Lemma \ref{lem: auxiliary lemmas on indisc}(\ref{lem_item: ramsey for arrays}), Ramsey  and compactness.
\end{proof}

\begin{definition}
We say that a (partial) type $p(x)$ over $A$ is \emph{$\NTP_2$-determined} if there is $\Phi \subseteq p$ closed under conjunction, such that $\Phi(x) \vdash p(x)$ and such that for every $\phi(x,a) \in \Phi$, $\phi(x,y)$ is $\NTP_2$. 
\end{definition}

\begin{lemma}
\label{L:NTP2determined} 
Let $(a_{ij})_{i,j \in \omega}$ be a strongly indiscernible array, $\phi(x,y)$ a formula and let $c \models \{ \phi(x,a_{i0}) \}_{i \in \omega}$, moreover assume that the sequence of rows $(\bar{a}_i)_{i \in \omega}$ is indiscernible over $c$. Assume that $p(x,a_{00})=\tp(c/a_{00})$ is $\NTP_2$-determined. Then $\{ \phi(x,a_{0j}) \}_{j \in \omega}$ is consistent.
\end{lemma}
\begin{proof}
By the choice of $c$ we have $\phi(x,a_{00}) \in p(x,a_{00})$, then by compactness (and as $\Phi$ is closed under conjunctions) there is some $\psi(x,a_{00}) \in \Phi(x)$ such that $\psi(x,a_{00}) \vdash \phi(x,a_{00})$. By strong indiscernibility it follows that $\psi(x,a_{ij}) \vdash \phi(x,a_{ij})$ for all $i,j \in \omega$. Note also that  $c \models \bigcup_{i \in \omega} \{ p(x,a_{i0})\}$, so in particular $c \models \{ \psi(x,a_{i0}) \}_{i\in \omega}$. As $\psi(x,z)$ is $\NTP_2$, it follows that for some $i \in \omega$ the set $\{ \psi(x,a_{ij}) \}_{j \in \omega}$ is consistent, so by strong indiscernibility the set $\{ \psi(x,a_{0j}) \}_{j \in \omega}$ is consistent. But this implies that $ \{ \phi(x,a_{0j}) \}_{j \in \omega} $ is consistent.
\end{proof}

\section{Preservation of $\NTP_2$}\label{S:Main}
In this section we prove the main results of the paper, concerning the preservation of $\NTP_2$ 
in various $\sigma$-henselian valued difference fields. We first prove a general preservation 
result and then apply this in various contexts.

\subsection{A general preservation result} \label{sec: general theorem}

\begin{theorem}\label{T:Main}
Let ${\cal K}=(K,k,\Gamma)$ be an $\ac$-valued difference fields of residue characteristic 0. Assume that $T=\mathrm{Th}({\cal K})$ eliminates $\vf$-quantifiers, and that both the residue field $k$ (as a difference field) and the value group $\Gamma$ (as an ordered difference group) are $\NTP_2$. 
Then, ${\cal K}$ is $\NTP_2$.
\end{theorem}

\begin{lemma}
\label{L: QFisNIP}
Let $\phi(x,y)$ be a quantifier-free formula (in the language $\L_{k,\Gamma,\sigma}\cup\{\ac\}$). Then  it is $\NIP$  in every $\ac$-valued difference field of residue characteristic 0.
\end{lemma}
\begin{proof}
We may write $\phi(x,y)$ as $\psi(x,\sigma(x),\ldots,\sigma^n(x),y, \sigma(y), \ldots, \sigma^n(y))$, where the formula $\psi(x_0,x_1,\ldots,x_n,y_0,y_1,\ldots, y_n)$ is quantifier-free in the language of $\ac$-valued fields $\L_{\rf,\vg}\cup\{\ac\}$. 

\begin{claim*}
Every $\ac$-valued field in residue characteristic 0 embeds into an algebraically closed $\ac$-valued field.
\end{claim*}

\begin{proof}[Proof of the claim]
This should be well known. We could not find a reference, and so we give a proof. Let ${\cal K}=(K,k,\Gamma)$ be an $\ac$-valued field of residue characteristic 0. 
The $\ac$-map (uniquely) extends to the henselisation of ${\cal K}$, so we may assume that ${\cal K}$ is henselian. We may 
now argue as in the proof of Fact \ref{F:ModComp}. It follows from Pas' theorem (Fact \ref{F:Pas-Thm}) that
${\cal K}\equiv(k((\Gamma)),k,\Gamma)$, where the latter is endowed with the standard $\ac$-map. Any Hahn series field embeds 
(as an $\ac$-valued field) in an algebraically closed Hahn series field. The result follows.
\end{proof}

By a result of Delon \cite{Del81} mentioned in the introduction, the theory $\ACVF_{0,0}$, in the language with $\ac$, is NIP.
It thus follows from the claim that $\psi(\bar{x},\bar{y})$ is $\NIP$ in every $\ac$-valued field of residue characteristic 0. Now, assume that $(a_i)_{i \in \omega}$ and $(b_s)_{s \subseteq \omega}$ from some $\ac$-valued difference field ${\cal K}$ are such that $\mathcal{K}\models \phi(a_i,b_s) \Leftrightarrow i \in s$. But then, letting $\bar{a}_i = ( a_i, \sigma(a_i), \ldots, \sigma^n(a_i)), \bar{b}_s = (b_s, \sigma(b_s), \ldots, \sigma^n(b_s))$, we get $\mathcal{K} \models \psi(\bar{a}_i,\bar{b}_s) \Leftrightarrow i \in s$. This contradicts the fact that $\psi(\bar{x},\bar{y})$ is $\NIP$.
\end{proof}

\begin{proof}[Proof of Theorem \ref{T:Main}]
We fix some monster model $\M \models T$. Suppose there is a formula $\phi(x,y)$ with $\TP_2$. By Fact \ref{F:1Variable}, we may assume that $|x|=1$. As the induced structures on $\vg$ and on $\rf$ are $\NTP_2$,  combining Lemma \ref{lem: very indisc witness of TP2} and Lemma \ref{lem: finding indiscernible row} we may assume that $x$ is 
a variable from the valued field sort $\vf$. 

Let $(a_{ij})_{i,j\in\omega}$ be an array witnessing that $\phi(x,y)$ has $\TP_2$, and let $a$ be a realisation of the first column, i.e.\ $a\models\bigwedge_{i\in\omega}\phi(x,a_{i0})$ 
. By Lemma \ref{lem: very indisc witness of TP2} we may assume that:

\begin{itemize}
\item $(a_{ij})_{i,j\in\omega}$ is a strongly indiscernible array;
\item the sequence of rows $(\a_i)_{i\in\omega}$ is $a$-indiscernible.
\end{itemize}

In our proof, we will successively construct arrays $(a^{\alpha}_{ij})$, where $\alpha$ is an ordinal $\leq \omega$ and $a^{\alpha}_{ij}$ is a countable tuple from $\M$, starting with 
$a^{0}_{ij}=a_{ij}$, and such that, 
for any $\beta>\alpha$, there is a decomposition $(a^{\beta}_{ij})=(a^*_{ij}b_{ij}^*)$ satisfying 
\begin{enumerate}[(1)]
\item $(a^\beta_{ij})_{i,j\in\omega}$ is a strongly indiscernible array;\label{Enu:VI}
\item the sequence of rows $(\overline{a}^\beta_i)_{i\in\omega}$ is $a$-indiscernible, and
\item $\overline{a}^*_{i}\equiv_{a^{\alpha}_{i0}}\overline{a}^{\alpha}_{i}$ for all $i\in\omega$\label{Enu:Cons}.
\end{enumerate}

It follows from (\ref{Enu:Cons}) that the the first column is just an extension of the original one, and that in particular we still have 
$a\models\bigwedge_{i\in\omega}\phi(x,a^\beta_{i0})$. Also by (\ref{Enu:Cons}), the rows $\{\phi(x,a^\beta_{ij})\}_{j\in\omega}$ are still inconsistent. As a 
consequence, for any $\beta$, the array $(a^\beta_{ij})$ witnesses that $\phi$ has $\TP_2$. 

Even though only the first column $(a^\alpha_{i0})_{i\in\omega}$ has to be a subtuple of $(a^\beta_{i0})$, we will say, somewhat inaccurately, that we \emph{extend} the array, when we pass from 
$(a^\alpha_{ij})$ to $(a^\beta_{ij})$. An extension of arrays will be called \emph{good} if it satisfies the properties (\ref{Enu:VI})--(\ref{Enu:Cons}) above.

\smallskip

The construction will be done in steps, following the back-and-forth system one may infer from (2) in 
Lemma \ref{L:Back-and-forth-QE}. There will be two kinds of successor steps: \emph{auxiliary steps}, where Lemma \ref{lem: adding carelessly} is used to extend the array $(a_{ij}^\alpha)$ carelessly, to add new parameters; \emph{treating steps}, 
where the Array Extension Lemma (Lemma \ref{prop: adding carefully}) is used to extend the array $(a_{ij}^\alpha)$ carefully, respecting partial information from $\tp(a/a_{00}^\alpha)$ coming from 
the $\NTP_2$ sorts $\vg$ and $\rf$. The final step dealing with an immediate extension will follow from Lemma \ref{L:NTP2determined}.

If $(a_{ij}^\alpha)$ have been constructed for all $\alpha < \omega$ such that $(a_{ij}^\beta)$ is a good extension of $(a^\alpha_{ij})$ for all $\alpha<\beta<\omega$, then
we may find an array $(a_{ij}^\omega)$ with $a_{i0}^\omega=\bigcup_{\alpha<\omega}a_{i0}^\alpha$ for all $i \in \omega$ and such that $(a_{ij}^\omega)$ is a good extension of $(a^\alpha_{ij})$, for all $\alpha<\omega$. Indeed, this follows 
from compactness, as properties (\ref{Enu:VI})--(\ref{Enu:Cons}) are type-definable in the variables $(x^\omega_{ij})_{i,j\in\omega}$.

\medskip

\begin{enumerate}[(I)]
\item\label{alg-step} Given $(a^\alpha_{ij})$, there is a good extension $(a^{\alpha+1}_{ij})$ such that $(a^{\alpha+1}_{00})$ enumerates a substructure 
${\cal K}^{\alpha+1}=(K^{\alpha+1},k^{\alpha+1},\Gamma^{\alpha+1})$ where both $K^{\alpha+1}$ and $k^{\alpha+1}$ are difference fields and $\Gamma^{\alpha+1}$ is an ordered $\Z[\sigma]$-module.

\noindent
[By Lemma \ref{lem: adding carelessly}.]

\medskip

In what follows, we may always assume that $a^{\alpha}_{00}$ is as in the conclusion of step \ref{alg-step}. To ease the notation, we write $a^{\alpha}_{00}=(K,k,\Gamma)$ and 
$a^{\alpha+1}_{00}=(K',k',\Gamma')$. Let $L:=K\langle a\rangle$, and $L':=K'\langle a\rangle$. Recall that $a\models\bigwedge_{i\in\omega}\phi(x,a_{i0})$ .
\medskip
\item \label{res-step} Given $(K,k,\Gamma)$, there is a good extension such that $k_{K'} \supseteq k$. 

\noindent
[By Lemma \ref{lem: adding carelessly}.]

\medskip
\item \label{val-step} Given $(K,k,\Gamma)$, there is a good extension such that $\Gamma_{K'} \supseteq \Gamma$. 

\noindent
[By Lemma \ref{lem: adding carelessly}.]
\medskip
\item \label{inert-step} Given $(K,k,\Gamma)$, there is a good extension such that $\ac(L)\subseteq k'$. 

\noindent
[By Lemma \ref{prop: adding carefully} with $b = \ac(L)$, as $\rf$ is stably embedded (Lemma \ref{L:QE-consequences}(1)), and $\NTP_2$ by assumption.]
\medskip
\item \label{ram-step} Given $(K,k,\Gamma)$, there is a good extension such that $\Gamma_L\subseteq \Gamma'$. 

\noindent
[By Lemma \ref{prop: adding carefully}, as $\vg$ is stably embedded (Lemma \ref{L:QE-consequences}(1)), and $\NTP_2$ by assumption.]

\medskip
Iterating steps \ref{alg-step}-\ref{ram-step} and passing to the limit, we may thus construct a good extension $(a^\omega_{ij})$ such that $(K,k,\Gamma)$ is coming from an $\ac$-valued difference field (i.e. $k=k_K$ and 
$\Gamma=\Gamma_K$) and such that $L/K$ is an immediate extension.

\medskip

\item As $\ac(K)\subseteq k_K$, we may apply Lemma \ref{L:QE-consequences}(2), and so $\tp(a/K)$ is determined by its quantifier-free part. By Lemma \ref{L: QFisNIP} every quantifier-free formula is $\NIP$, so in particular is $\NTP_2$. Thus, $\tp(a/K)$ is $\NTP_2$-determined. From Lemma \ref{L:NTP2determined} it then follows that $\{\phi(x,a^{\omega}_{0j}) \}_{j \in \omega}$ is consistent. But $\{\phi(x,a^\omega_{ij})\}_{i,j<\omega}$ is a witness that $\phi(x,y)$ is $\TP_2$ --- a contradiction.\qedhere

\noindent

\end{enumerate}
\end{proof}

In fact, it is easy to see that the previous proof gives the following stronger result. 

\begin{remark}
With the notations and assumptions of Theorem \ref{T:Main}, suppose that $T_r'\supseteq \mathrm{Th}(k,\overline{\sigma})$ and $T_v'\supseteq \mathrm{Th}(\Gamma,\sigma)$ are expansions which are both $\NTP_2$. Then, $T':=T\cup T_r'\cup T_v'$ eliminates $\vf$-quantifiers and is $\NTP_2$.\qed
\end{remark}

\subsection{Applications to $\sigma$-henselian valued difference fields}\label{sec: applications of the main theorem}
We start with the contractive case. We fix a completion $T$ of $T_0^{\ac}$, so $T$ is of the form $T_0^{\ac}\cup T_r\cup T_v$, where $T_r$ is a complete theory of difference fields of characteristic 0 (which has to be assumed 
to be linearly difference closed by Remark \ref{R:LinDiffCl}) and $T_v$ is a complete theory of ordered $\Z[\sigma,\sigma^{-1}]$-modules.

\begin{theorem}\label{T:Main-omega}
Assume that both $T_r$ and $T_v$ are $\NTP_2$. Then, $T$ is $\NTP_2$.
\end{theorem}

\begin{proof}
Combine Theorem \ref{T:Main} with Fact \ref{F:sigmaAKE}(1).
\end{proof}

\begin{corollary}\label{C:VFA-NTP2}
Every completion of $\VFA$ is $\NTP_2$.
\end{corollary}

\begin{proof}
Every completion of $\ACFA_0$ is simple by \cite{ChHr99}, so in particular it is $\NTP_2$ by Fact \ref{F:NIPSimpleIsNTP2}. The theory $\IncDODG$ is $o$-minimal 
by Fact \ref{F:ValInc}, so in particular it is $\NIP$ and thus $\NTP_2$, by Fact \ref{F:NIPSimpleIsNTP2}. We may thus conclude by Theorem \ref{T:Main-omega}.
\end{proof}

Next, we consider the isometric case.
\begin{theorem}\label{T:PresNTP2Isom}
Let ${\cal K}=(K,k,\Gamma,\sigma,\val,\ac)\models S_0^{ac}$, i.e., ${\cal K}$ is a $\sigma$-henselian valued difference field of residue characteristic 0, $\sigma$ is an isometry and there 
are enough constants. Then $\Th({\cal K})$ is $\NTP_2$ if and only if $\Th(k,\overline{\sigma})$ is $\NTP_2$.
\end{theorem}

\begin{proof}
Combine Theorem \ref{T:Main} with Fact \ref{F:IsomAKE}. Note that in the isometric case, as well as in the case of henselian valued fields (Fact \ref{F:NTP2ofValFields}), there is no condition on 
$\Gamma$, since $\sigma_{\Gamma}=\mathrm{id}$ in these cases, and so the induced structure is that of an ordered abelian group. By a result of Gurevich and Schmidt \cite{GuSc84}, any ordered 
abelian group is NIP.
\end{proof}

Finally, our result applies to valued fields without an automorphism in the language as well.

\begin{corollary}[\cite{Che12}]\label{F:NTP2ofValFields}
Let ${\cal K}= (K,\Gamma,k)$ be a henselian valued field of residue characteristic $0$ in the Denef-Pas language. Assume that the theory of $k$ is $\NTP_2$. Then the theory of ${\cal K}$ is $\NTP_2$.\end{corollary}
\begin{proof}
Any $\ac$-valued field may be considered as an $\ac$-valued difference field, with $\sigma=\mathrm{id}$.
As $(K,\Gamma,k)$ eliminates field quantifiers (Fact \ref{F:Pas-Thm}), Theorem \ref{T:Main} applies.
\end{proof}

We remark that the proof from \cite{Che12} also shows that strength is preserved, see Section \ref{sec:Further model theoretic properties of VFA}.

\medskip

\begin{remark}\label{R:Mult-NTP2}
One may show in the same way that in the multiplicative case from \cite{Pal10} (see Remark \ref{R:Mult-Case}), the 
valued difference field is $\NTP_2$, provided $\RV$ (with the induced structure) is $\NTP_2$.
\end{remark}

\section{Open problems}

In the last section we discuss some open problems, pose several questions and consider possible research directions around model-theoretic properties of valued difference fields.

\subsection{Further model theoretic properties of $\VFA$} \label{sec:Further model theoretic properties of VFA}

\begin{definition}
A theory is called \emph{strong} if there are no $(\phi_i(x,y_i), \bar{a}_i, k_i)$ with $\bar{a}_i = (a_{ij})_{j \in \omega}$ and $k_i \in \omega$ such that:
\begin{itemize}
\item $\{ \phi_i(x,a_{ij}) \}_{j \in \omega}$ is $k_i$-inconsistent, for all $i \in \omega$,
\item $\{ \phi_i(x,a_{i f(i)}) \}_{i \in \omega}$ is consistent for every $f: \omega \to \omega$.
\end{itemize}
\end{definition}

Strong theories were defined by Adler in \cite{AdlBurden}. They form a subclass of $\NTP_2$ theories which can be viewed as `super $\NTP_2$'. For more on strong theories and the related notion of \emph{burden} see \cite{Che12}.
\begin{question}
Is $\VFA$ strong?
\end{question}
The following remark implies that $\VFA$ is not of finite burden, as in a simple theory burden of a type equals the supremum of the weights of its completions \cite{AdlBurden}. 

\begin{remark}
\begin{enumerate}
\item Let $T$ be a simple theory. Assume that in $T$ there is a (type-) definable infinite field $F$ and a (type-)definable $F$-vector space $V$ of dimension $\geq n$. 
Then, there is a type $p(x)\vdash x\in V$ such that $\weight(p)\geq n$. 
\item Every completion of $\ACFA$ has a 1-type of weight $\geq n$, for any $n \in \omega$.
\end{enumerate}
\end{remark}

\begin{proof}
To prove (1), choose $v_1,\ldots,v_n\in V$ which are $F$-linearly independent. Choose any non-algebraic type $q(x)$ 
such that $q(x)\vdash x\in F$. Let $\b=(b_1,\ldots,b_n)$ be a sequence of independent realisations of $q$ (over some model $M$ containing the $v_i$'s). 
Let $b:=\sum_{i=1}^n b_i v_i$. Since $\b$ and $b$ are interdefinable over $M$, we compute (see e.g. \cite[Lemma 5.2.4]{WagnerBook}):
$\weight(\tp(b/M))=\weight(\tp(\b/M))=n\weight(q)\geq n$. 

The second part follows, considering $F:=\mathrm{Fix}(\sigma)$ and $V:=K$ and noting that the dimension of $V$ over $F$ is infinite.
\end{proof}

\

\begin{definition}
We say that a formula $\varphi\left(x,y\right)$ is \emph{resilient}
if it satisfies the following property:
\begin{itemize}
\item For any indiscernible sequences $\bar{a}=\left(a_{i}\right)_{i\in\mathbb{Z}}$
and $\bar{b}=\left(b_{i}\right)_{i\in\mathbb{Z}}$ such that $a_{0}=b_{0}$
and $\bar{b}$ is indiscernible over $a_{\neq0}$, if $\left\{ \varphi\left(x,a_{i}\right)\right\} _{i\in\mathbb{Z}}$
is consistent, then $\left\{ \varphi\left(x,b_{i}\right)\right\} _{i\in\mathbb{Z}}$
is consistent.
\end{itemize}
A theory is resilient if it implies that every formula is resilient.
\end{definition}
Resilient theories were introduced in \cite{CheBY} where it was observed
that:
\begin{remark}
\label{rem:NIPimpliesResilience}
\begin{enumerate}
\item Every formula in a simple theory is resilient.
\item Every $\NIP$ formula is resilient.
\item Every resilient theory is $\NTP_{2}$.
\end{enumerate}
\end{remark}

It is not known if there are $\NTP_{2}$
theories which are not resilient. 

\begin{conjecture}
An analog of Theorem \ref{T:Main} holds for resilience.
\end{conjecture}

An earlier version of the article contained a purported proof of this following the strategy of the proof for $\NTP_2$, but a flaw was pointed out by the referee.

Some further model theoretic properties of $\VFA$ are of interest, both for sets in the real sort and  in $M^{eq}$, and most importantly in the \emph{geometric sorts }from 
\cite{HHMStableDomination}:

\begin{question}
\begin{enumerate}
\item Is $\VFA$ extensible? I.e. is it true that for every small set $A$, every type $p(x) \in S(A)$ has a global extension which does not fork over $A$? Note that it is enough to check this property for $1$-types.

\item Is $\VFA$ low? I.e., is it true that for every formula $\phi(x,y)$ there is $k \in \omega$ such that for every indiscernible sequence $(a_i)_{i \in \omega}$, the set $\{\phi(x,a_i)\}_{i \in \omega}$ is consistent if and only if it is $k$-consistent?

\item Does $\VFA$ eliminate $\exists^{\infty}$?

\end{enumerate}
\end{question}

It seems tempting to try to develop a theory of \emph{simple} domination in $\VFA$ (parallel to stable domination from \cite{HHMStableDomination}). Some elements of the theory of simple types in $\NTP_2$ theories are developed in \cite{Che12}. 

\begin{question}
Is it true in $\VFA$ that a union of two stably embedded sets is stably embedded? Is it at least true for simple stably embedded sets?
\end{question}

\subsection{Ordered modules}
With a view on our main results, it would be interesting to know which ($\omega$-increasing) ordered difference 
groups are $\NTP_2$, or even $\NIP$. We will put this issue in a larger context. Let $R$ be an ordered ring. Recall 
that an \emph{ordered $R$-module} is an ordered abelian group $\langle M,0,+,<\rangle$ together with an action 
of $R$ by endomorphisms which is compatible with the orderings, i.e.\ such that $r\cdot m>0$ for all $r>0$ from $R$ and 
all $m>0$ from $M$. We consider ordered $R$-modules in the language $\L_{R- mod,<}=\{0,+,<\}\cup\{\lambda_r\mid r\in R\}$, where 
$\lambda_r$ is a unary function which is interpreted by the scalar multiplication by $r$.

\begin{question}
\begin{enumerate}
\item Are all ordered $R$-modules $\NIP$ (for all ordered rings $R$)?
\item More specifically, are all $\omega$-increasing ordered difference groups $\NIP$? (This corresponds to the case where $R=\Z[\sigma,\sigma^{-1}]$, see Section \ref{Sub:ODG}.)
\end{enumerate}
\end{question}

Recall that every module is stable (see e.g. \cite{Hod93}), and that every ordered abelian group is $\NIP$, by a result 
of Gurevich and Schmidt \cite{GuSc84}. Therefore, one might suspect a positive answer even to the first question. It 
seems that the answer to this question is unknown.

\smallskip

We now give a result which covers some easy cases. There are some similarities with work of Robinson and Zakon on (archimedean) ordered abelian groups \cite{RoZa60}.

\begin{proposition}\label{P:ordered-PID}
Let $R$ be an ordered ring. Assume the following conditions:
\begin{enumerate}
\item[(i)] $R$ is a principal ideal domain;
\item[(ii)] $R$ is densely ordered;
\item[(iii)] for every prime $\pi\in R$, the ideal $\pi R$ is dense in $R$, and
\item[(iv)] for every prime $\pi\in R$, $R/\pi R$ is infinite.
\end{enumerate}
  
  \smallskip
  
Let $T$ be the $\L_{R- mod,<}$-theory of $R$, considered as an ordered module over itself. Then the following holds:
\begin{enumerate}
\item A non-zero ordered $R$-module $M$ is a model of $T$ iff, for every prime $\pi\in R$,
\begin{enumerate}
\item[(a)] $\pi M$ is dense in $M$, and
\item[(b)] $M/\pi M$ is infinite.
\end{enumerate}
\item $T$ eliminates quantifiers in the language $\L_{R- mod,<}\cup\{ P_r, r\in R\}$, where $P_r(x):\Leftrightarrow \exists y \, r\cdot y=x$. 
\item $T$ is $\NIP$.
\end{enumerate}
\end{proposition}

\begin{proof}
It is a classical result (see, e.g.\ \cite{Pre88}) that in the class of $R$-modules (without the order), for a ring $R$ satisfying (i) and (iv), a non-zero $R$-module $M$ 
is elementarily equivalent to $R$ (as an $R$-module) iff $M$ is torsion free and $M/\pi M$ is infinite for every prime $\pi\in R$. Moreover, $T\upharpoonright_{\L_{R- mod}}$ 
eliminates quantifiers in the language $\L_{R- mod}\cup\{ P_r, r\in R\}$.

Now put $\L=\L_{R- mod,<}\cup\{ P_r, r\in R\}$, and let $T'$ be the $\L$-theory of non-zero ordered $R$-modules satisfying (a) and (b). We will show that $T'$ 
eliminates quantifiers in $\L$. This will prove (1) and (2). We use a standard back-and-forth argument. Let $M$ and $N$ be two models of $T'$, with 
$N$ sufficiently saturated. Assume that $f:A\cong B$ is an $\L$-isomorphism between finitely generated substructures $A\subseteq M$ and $B\subseteq N$. 

Now let $\tilde{a}\in M$. If $\tilde{a}$ is in the divisible hull of $A$, then $f$ extends (even uniquely) to an $\L$-isomorphism $g:A+R\tilde{a}\rightarrow B+R\tilde{b}$ for some 
$\tilde{b}\in N$. 

We now assume that $\tilde{a}$ is not in the divisible hull of $A$. By the elimination of quantifiers down to $\L_{R- mod}\cup\{ P_r, r\in R\}$ mentioned in the first paragraph, there is $\tilde{b}'\in N$ 
such that $\tilde{a}\mapsto \tilde{b}'$ defines an extension of $f$ to an $\L_{R- mod}\cup\{ P_r, r\in R\}$-isomorphism $g':A+R\tilde{a}\cong B+R\tilde{b}'$. Of course, $g'$ might 
not preserve the order. We will show that there is $d\in N$ such that $d$ is divisible by every non-zero $r\in R$ and such that $\tilde{a}\mapsto \tilde{b}=\tilde{b}'+d$ defines an extension of $f$ to an $\L$-isomorphism $g: A+R\tilde{a}\cong B+R\tilde{b}$. 

Let $Q(R)$ be the field of fractions of $R$. Recall that if $C$ is an ordered $R$-module, the order extends uniquely to $Q(C)=C\otimes_R Q(R)$ so that $Q(C)$ is an ordered $R$-module extending $C$. Now by assumption we have $R\tilde{a}\cap A=(0)$, so $\tilde{a}$ determines a cut $(L,R)$ in $Q(A)$. Let $(f(L),f(R))$ be the cut over $Q(B)$ induced by 
$f$. Over $Q(B+R\tilde{b}')$, we may look at the (consistent) partial type $\pi(x)$ given by $f(L)-\tilde{b}'<x<f(R)-\tilde{b}'$. Note that the density assumption (a) implies that $rN$ is dense in $N$ for any non-zero $r$. 
Thus, by saturation of $N$, we may find $d\in N$ such that 
$d$ is divisible by any non-zero $r\in R$ and such that $d\models \pi$. By construction, $\tilde{b}=\tilde{b}'+d$ is as we want. (In particular $\tilde{b}$ is not in the divisible hull of $B$.)

(3) follows from (2), taking into account that it is enough to show that any formula $\phi(x,y)$ with $x$ a singleton is NIP, and that NIP formulas are closed under Boolean combinations. 
\end{proof}

We note that condition (ii) in the proposition actually follows from (iii) and (i).

\begin{corollary}\label{C:ExampleNIP}
Consider the ordered field $\Q(\sigma)$, with $\sigma\gg1$. Then every ordered subring of $\Q(\sigma)$ containing $\Q[\sigma,\sigma^{-1}]$ is $\NIP$, 
considered as an ordered module over itself. In particular, the ordered difference group $\Q[\sigma,\sigma^{-1}]$ is $\NIP$.
\end{corollary}

\begin{proof}
We need to show that the hypotheses of Proposition \ref{P:ordered-PID} hold. So let $R$ be an ordered ring with 
$\Q[\sigma,\sigma^{-1}]\subsetneq\Q(\sigma)$.

If $A$ is a PID with field of fractions $K$, then 
every ring $B$ with $A\subseteq B\subseteq K$ is a PID, as $B$ is necessarily a localisation of $R$. This shows that $R$ is a PID.

Property (ii) holds since $(R,+)$ is a divisible ordered abelian group. 

For (iii), note that $\Q[\sigma,\sigma^{-1}]$ is dense in $\Q((\sigma))$. More generally, for any $0\neq s\in\Q((\sigma))$, 
the set $s\Q[\sigma,\sigma^{-1}]$ is dense in $\Q((\sigma))$. In particular, for any prime $\pi$ of $R$, $\pi R$ is dense in $\Q((\sigma))$, and so in $R$ as well. 

(iv) is clear, as $\pi R$ is a proper $\Q$-vector subspace of $R$.
\end{proof}

\bibliography{VFA_NTP2}

\end{document}